\documentclass[12pt]{amsart}
\usepackage{etex}
\usepackage{amssymb,amsmath,amsfonts,fancyhdr,amsthm,verbatim,stmaryrd,hyperref}
\usepackage{graphicx,supertabular,paralist}
\usepackage[dvips,all]{xy}
\UseComputerModernTips

\usepackage{caption}
\usepackage{subcaption}
\usepackage{tikz}
\usepackage{pst-plot}
\usepackage{pst-node,pst-text,pst-3d,pstricks}
\usetikzlibrary{patterns}

\numberwithin{equation}{section}
\newtheorem{theorem}{Theorem}[section]

\newtheorem{proposition}[theorem]{Proposition}
\newtheorem{corollary}[theorem]{Corollary}

\theoremstyle{definition}
\newtheorem{definition}[theorem]{Definition}
\newtheorem{def-prop}[theorem]{Definition-Proposition}
\newtheorem{remark}[theorem]{Remark}
\newtheorem{example}[theorem]{Example}

\def\FF{\mathbb{F}}
\def\NN{\mathbb{N}}

\def\QQ{\mathbb{Q}}
\def\RR{\mathbb{R}}
\def\ZZ{\mathbb{Z}}

\def\A{{\mathcal A}}

\def\E{{\mathcal E}}

\def\P{{\mathcal P}}

\def\a{{\bf a}}

\def\1{{\bf 1}}
\def\0{{\bf 0}}

\begin{document}


\title{Normal 0-1 polytopes}

\author{Huy T\`ai H\`a}
\address{Tulane University \\ Department of Mathematics \\
6823 St. Charles Ave. \\ New Orleans, LA 70118, USA}
\email{tha@tulane.edu}
\urladdr{http://www.math.tulane.edu/$\sim$tai/}

\author{Kuei-Nuan Lin}
\address{Penn State Greater Allegheny\\ Academic Affairs \\
4000 University Dr. \\ McKeesport, PA 15132, USA}
\email{kul20@psu.edu}
\urladdr{http://www.personal.psu.edu/$\sim$kul20/}

\thanks{H\`a is partially supported by a grant \#279786 from the Simons Foundation.}

\begin{abstract}
We study the question of when 0-1 polytopes are normal or, equivalently, having the integer decomposition property. In particular, we shall associate to each 0-1 polytope a labeled hypergraph, and examine the equality between its Ehrhart and polytopal rings via the combinatorial structures of the labeled hypergraph.
\end{abstract}

\maketitle


\section{Introduction} \label{sec.intro}

A 0-1 polytope in $\RR^n$ is the convex hull of a finite set of (0,1)-vectors, that is, the convex hull of a subset of the vertices of the regular cube $\{0,1\}^n$. These polytopes, though simply defined, are important in combinatorial optimization. Special classes of 0-1 polytopes have been much studied, for instance, the traveling salesman polytopes (cf. \cite{ABCC, GrPa}) and cut polytopes (cf. \cite{BM, DL}).

An integral convex polytope $\P \subseteq \RR^n$ is said to be \emph{normal} if
$$(j\P \cap \ZZ^n) + (l\P \cap \ZZ^n) = (j+l)\P \cap \ZZ^n \ \forall \ j, l \in \NN.$$
Such polytopes are also referred to as having the \emph{integer decomposition property} (cf. \cite{CHHH}).
In this paper, we investigate the question of which 0-1 polytopes are normal. Equivalently, we look at graded algebras associated to a 0-1 polytope $\P$, namely the \emph{Ehrhart ring} $\A[\P]$ and \emph{polytopal ring} $k[\P]$, and study the question of when these two rings are the same. Roughly speaking, as $k$-vector spaces, the Ehrhart ring $\A[\P]$ has a basis consisting of monomials whose exponents are in $\{ (\a,t) \in (t\P \cap \ZZ^n) \times \{t\} \subseteq \ZZ^n \times \NN\}$ while the polytopal ring $k[\P]$ has a basis consisting of monomials whose exponents are nonnegative integral combinations of points in $(\P \cap \ZZ^n) \times \{1\} \subseteq \ZZ^n \times \NN$.

In general, $k[\P]$ is a subalgebra of $\A[\P]$, and $\A[\P]$ is normal and integral over $k[\P]$. Thus, as a consequence of the normality of $\P$, the polytopal ring $k[\P]$ would be normal. The polytopal ring $k[\P]$ is a subalgebra generated by monomials in a polynomial ring. Subalgebras generated by squarefree monomials associated to graphs and hypergraphs (also referred to as \emph{toric} rings) have attracted much attention in recent years. They have been investigated from various angles with various applications; for instance, in algebraic statistics (cf. \cite{GP, PS}), and in coding theory (cf. \cite{RSV,SVPV}), in characterizing when symbolic and ordinary powers of squarefree monomial ideals are equal (cf. \cite{HHTZ}), in investigating algebraic properties of graded algebras (cf. \cite{OK}), and in connection to graph and hypergraph theory (cf. \cite{HO, SVV0, SVV}). The normality of these toric rings is not only an important algebraic property but also closely connected to a long-standing conjecture in combinatorial optimization, the Conforti-Cornu\'ejols conjecture, which states the equivalence between the \emph{packing} and \emph{max-flow-min-cut} properties of a dual integer linear programming system (cf. \cite{DRV, DV, FHM, HM, GVV, HHTZ, MOV}). See particularly \cite[Conjectures 3.14 and 3.15]{DV} for an example of how the equality $\A[\P] = k[\P]$ is related to the Conforti-Cornu\'ejols conjecture. See also \cite[Corollary 1.6]{HHTZ} for a precise algebraic interpretation of the Conforti-Cornu\'ejols conjecture.

Our method is to identify the vertices of a 0-1 polytope $\P$ with squarefree monomials (which we view as the generators for an ideal $I$), and then to make use of the correspondence ($I \mapsto H_I$ and $I_H \mapsfrom H$) between squarefree monomial ideals and \emph{labeled hypergraphs} that was recently introduced by the second author and McCullough \cite{LM}. We shall examine combinatorial structures in a labeled hypergraph $H$ that guarantee/obstruct the normality of $\P$. Our work exhibits a nice interplay between discrete mathematics, combinatorial optimization and commutative algebra.

Simple hypergraphs are also known as \emph{Sperner systems} or \emph{clutters}. Clutters for which corresponding \emph{edge polytopes} are normal are called \emph{Ehrhart clutters}. Conditions for a clutter to be Ehrhart were discussed and combinatorial aspects of Ehrhart clutters were investigated in \cite{MOV}. Since, as we shall note later on, labeled hypergraphs are the \emph{duals} of resulting clutters, our study can be seen as addressing the dual problems to those in \cite{MOV}.

The starting point of our work is a simple observation, Proposition \ref{thm.nooddcycle}, in which it is shown that if $H = H_I$ is a \emph{balanced} labeled hypergraph and the generators of $I$ are of the same degree then the corresponding polytope $\P$ is normal. Our focus is thus on labeled hypergraphs that are not balanced, i.e., hypergraphs that do contain \emph{special} odd cycles.

Our first main result, Theorem \ref{thm.oneoddcycle}, gives a necessary and sufficient condition for the normality of a 0-1 polytope $\P$ in the case where the 1-skeleton of the corresponding labeled hypergraph $H$ is connected and contains odd cycles. Moving beyond this case, our next main result, Theorem \ref{thm.bi-color}, gives a necessary condition for the normality of $\P$ in the case where the 1-skeleton of $H$ is connected but does not necessarily contain odd cycles. In fact, we shall give a sufficient condition for $\A[\P] \not= k[\P]$ when the 1-skeleton of $H$ is connected but does not necessarily contain odd cycles.

The threshold of our work in the case where the 1-skeleton of the labeled hypergraph is not necessarily connected is Theorem \ref{thm.subhypergraph}, in which it is proved that if $H'$ is a \emph{minor} of $H$ (corresponding to polytopes $\P'$ and $\P$, respectively) and $\P'$ is not normal then neither is $\P$. Theorem \ref{thm.subhypergraph} allows us to restrict our attention to ``minimal'' hypergraphs whose corresponding polytopes are not normal. Inspired by the \emph{bow-tie} graphs introduced in \cite{SVV} and the \emph{odd cycles condition} discussed in \cite{HO}, we introduce the notion of an \emph{exceptional} pair of odd cycles. Our last main result, Theorem \ref{thm.exceptional}, shows that if $H$ contains an exceptional pair of odd cycles then the corresponding polytope $\P$ is not normal.

The paper is outlined as follows. In the next section, we collect notations and terminology used in the paper. In particular, we recall the definition of labeled hypergraphs and the correspondence between labeled hypergraphs and squarefree monomial ideals. In Section \ref{sec.red}, we recall a combinatorial criterion for the normality 0-1 polytopes, and provide a number of reduction results that allow us to restrict our attention to subhypergraphs. We prove Theorem \ref{thm.subhypergraph} in this section. The last section of the paper is devoted to our main results. Theorems \ref{thm.oneoddcycle}, \ref{thm.bi-color} and \ref{thm.exceptional} are proved in this section.

\noindent{\bf Acknowledgement.} The authors would like to thank anonymous referees for many useful suggestions/comments in improving the presentation of our paper. Many of our examples are computed using Macaulay 2 \cite{M2}.


\section{Preliminaries} \label{sec.prel}

In this section, we collect basic notation and terminology used in the paper. We follow standard texts in the area \cite{BH, CLS, E, S, SH}.

Throughout the paper, $\NN$ will denote the set of nonnegative integers (including 0). Let $k$ be a field, and let $R = k[x_1, \dots, x_n]$ and $S = k[x_1, \dots, x_n,y]$ be polynomial rings over $k$. We shall first recall the notion of a normal polytope.

\begin{definition} \label{def.NP}
An integral convex polytope $\P \subseteq \RR_{\ge 0}^n$ is called \emph{normal} if
$$(j\P \cap \ZZ^n) + (l\P \cap \ZZ^n) = (j+l)\P \cap \ZZ^n \ \forall \ j, l \in \NN.$$
\end{definition}

Associated to an integral convex polytope are naturally graded algebras, the Ehrhart and polytopal rings.

\begin{definition} \label{def.Ehrhart}
Let $\P \subseteq \RR_{\ge 0}^n$ be an integral convex polytope.
\begin{enumerate}
\item The \emph{Ehrhart ring} of $\P$, denoted by $\A[\P]$, is defined to be the subalgebra $\A[\P] = \bigoplus_{t \ge 0} \A[\P]_t$ in $S$, where $\A[\P]_t$ is the $k$-vector space spanned by the monomials $\{ x^\a y^t ~|~ \a \in t\P \cap \ZZ^n\}$ (here $t\P = \{t\a ~|~ \a \in \P\}$).
\item The \emph{polytopal ring} of $\P$, denoted by $k[\P]$, is defined to be the subalgebra $k[x^\a y ~|~ \a \in \P \cap \ZZ^n]$ in $S$.
\end{enumerate}
\end{definition}

\begin{remark} \label{rmk.NP}
It can be seen that $\P$ is a normal polytope if and only if $\A[\P] = k[\P]$. In fact, most of our results are stated in terms of this equality. In general, we have $k[\P] \subseteq \A[\P]$, and $\A[\P]$ is normal and integral over $k[\P]$ (cf. \cite{BH}).
\end{remark}

\begin{remark}
Our definition of normal polytopes is that of \cite{CLS}. This is not to be confused with the normality of the polytopal ring $k[\P]$, as used by some other authors, for instance, \cite{BGT}. We shall see later, in Remark \ref{rmk.referee}, that the normality of $k[\P]$ does not necessarily imply the normality of $\P$.
\end{remark}

As pointed out in \cite{MOV}, the equality $\A[\P] = k[\P]$ occurs if and only if $\{(\a_1, 1), \dots, (\a_s,1)\}$, where $\{\a_1, \dots, \a_s\}$ are the vertices of $\P$, is a \emph{Hilbert basis}. Thus, this property can be verified using computational packages, such as Normaliz \cite{BI}. The motivation of our work in this paper, on the other hand, is to identify new families of and to find new necessary and/or sufficient conditions for normal 0-1 polytopes.

From now on, let $\P$ be a 0-1 polytope. Then $\P$ contains no integral points except at its vertices, and so
$k[\P] = k[x^\a y ~|~ \a \text{ is a vertex of } \P].$
In this case, we can identify the vertices of $\P$ with generators of a squarefree monomial ideal
$$I_\P = \langle x^\a ~|~ \a \text{ is a vertex of } \P \rangle \subseteq k[x_1, \dots, x_n].$$
This squarefree monomial ideal will correspond to another combinatorial object, namely, a labeled hypergraph.

We shall now recall the notion of labeled hypergraphs introduced in \cite{LM}.

\begin{definition} \label{def.LH}
Let $\A$ be an alphabet, let $V$ be a finite set, and let $\P(V)$ denote its power set. A \emph{labeled hypergraph} $H$ on $V$ with alphabet $\A$ is consists of a function $f: \A \rightarrow \P(V)$. Elements of $V$ are the \emph{vertices}, elements of $X = \{a \in \A ~|~ f(a) \not= \emptyset\}$ are the \emph{labels}, and elements of $\E = \text{Im} f$ are the \emph{edges} in $H$. We often write $H = (V,f)$; the sets $\A$, $X$ and $\E$ are understood from the definition of $f$.
\end{definition}

For $E \in \E$, elements $a \in \A$ for which $f(a) = E$ are called \emph{labels} of $E$. For edges $E,F \in \E$, we say that $E$ is a \emph{subedge} of $F$ if $E \subseteq F$. An edge $E$ in $H$ is called \emph{simple} if it does not contain any proper subedges.
A vertex $v \in V$ is called a \emph{closed} vertex if $\{v\} \in \E$; otherwise $v$ is said to be an \emph{open} vertex.

\begin{definition} \label{def.subhypergraph}
Let $H = (V,f)$ be a labeled hypergraph and let $W \subseteq V$ be a subset of the vertices. The \emph{induced subhypergraph} of $H$ on $W$ is the labeled hypergraph $H|_W = (W, g)$, where $g$ is obtained by composing $f$ with the restriction map $\P(V) \rightarrow \P(W)$. For an edge $E$ in $H$, we call $E \cap W$ the \emph{contraction} of $E$ on $W$.
\end{definition}

\begin{definition} \label{def.minor}
Let $H = (V,f)$ be a labeled hypergraph.
\begin{enumerate}
\item Let $E$ be an edge in $H$. The \emph{deletion} $H \setminus E$ of $E$ from $H$ is obtained by removing $E$ and all vertices belonging to $E$ from $H$, i.e., $H \setminus E = H|_{V \setminus E}$.
\item A hypergraph obtained from $H$ by a sequence of deletions of edges is called a \emph{minor} of $H$.
\end{enumerate}
\end{definition}

It has been shown (cf. \cite{HO, SVV}) that the cycle structure of a graph has a strong connection to many algebraic properties of the toric ring of its edge ideal. We shall see that this is also the case for labeled hypergraphs.

\begin{definition} \label{def.cycle}
Let $H$ be a labeled hypergraph.
\begin{enumerate}
\item A \emph{cycle} in a labeled hypergraph $H$ is an alternating sequence of distinct vertices and edges in $H$, namely $v_1, E_1, v_2, E_2, \dots, v_m, E_m, v_{m+1} = v_1$, so that $v_i, v_{i+1} \in E_i$ for all $i$. We call $v_1, \dots, v_m$ (and only those) \emph{vertices}, and call $E_1, \dots, E_m$ \emph{edges} of the cycle.
\item A cycle in $H$ is called \emph{special} if it has no edge that contains more than 2 vertices in the cycle.
\end{enumerate}
\end{definition}

The \emph{1-skeleton} of a labeled hypergraph $H$ is a graph in the classical sense whose vertices are vertices in $H$ and whose edges are 1-dimensional edges in $H$. An \emph{$l$-coloring} of a graph $G$ is an assignment of $l$ colors to the vertices in $G$ such that adjacent vertices receive different colors.

Associated to any squarefree monomial ideal, one can construct a labeled hypergraph.

\begin{definition} \label{def.LHideals}
Let $I \subseteq R = k[x_1, \dots, x_n]$ be a squarefree monomial ideals with minimal generators $\{F_1, \dots, F_m\}$. The \emph{labeled hypergraph} $H_I$ associated to $I$ is defined by taking $\A = \{x_1, \dots, x_n\}$, $V = \{1, \dots, m\}$ and $f: \A \rightarrow \P(V)$ to be the function $f(x) = \{j ~|~ x \text{ divides } F_j\}.$
\end{definition}

\begin{example} Let $R = k[a,b,\dots, y,z]$ and let $I = (afh, aefgij, bchij, dghij)$. The labeled hypergraph $H_I$ of $I$ is shown in Figure \ref{fig.LH}. It consists of 4 vertices and 4 labeled edges. In $H_I$, $1$ is an open vertex, and $2,3$ and $4$ are closed vertices (we often indicate a close vertex by a filled circle and an open vertex by an unfilled circle).
\end{example}

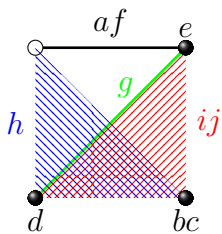
\begin{figure}[h]
\begin{center}
\begin{tikzpicture}

\draw [line width=1pt, color=green ] (2,-2)--(4,0);
\draw [line width=1pt  ] (2.1,0)--(4,0);

\path  (2,0)--(2,-2)  node [pos=.5, left ]{${\color{blue}{h}}$};
\path  (4,0)--(4,-2)  node [pos=.5, right ]{${\color{red}{ij}}$};
\path  (2,0)--(4,0)  node [pos=.5, above ]{$af$};
\path  (4,0)--(4,0)  node [pos=.5, above ]{$e$};
\path  (4,-2)--(4,-2)  node [pos=.5, below ]{$bc$};
\path  (2,-2)--(2,-2)  node [pos=.5, below ]{$d$};
\path  (2,-2)--(4,0)  node [pos=.6, above ]{${\color{green}{g}}$};
\path [pattern=north west lines, pattern  color=blue]   (2,0)--(4,-2)--(2,-2)--cycle;
\path [pattern=north east lines, pattern color=red]   (4,0)--(4,-2)--(2,-2)--cycle;

\draw  [shape=circle] (2,0) circle (.1);
\shade [shading=ball, ball color=black]  (4,0) circle (.1) ;
\shade [shading=ball, ball color=black]  (2,-2) circle (.1);
\shade [shading=ball, ball color=black]  (4,-2) circle (.1);

\end{tikzpicture}
\end{center}
\caption{Labeled hypergraph of a squarefree monomial ideal.}\label{fig.LH}
\end{figure}

Kimura {\it et. al.} \cite{KTY} introduced the \emph{unlabeled} version of $H_I$. This unlabeled version of $H_I$ coincides with the \emph{dual hypergraph} of the hypergraph whose edge ideal is $I$ (cf. \cite{Berge}). We choose not to use the notion of dual hypergraphs in order not to get confused with the \emph{Alexander dual}. 
Note that different squarefree monomial ideals may give the same unlabeled hypergraph. For instance, $I = (ac,bc)$ and $J = (acd,bcd)$ both give the unlabeled hypergraph on $\{1,2\}$ with edges $\{\{1\},\{2\},\{1,2\}\}$. Note also that not all unlabeled hypergraphs come from squarefree monomial ideals. It was pointed out in \cite{KTY} that unlabeled hypergraphs associated to squarefree monomial ideals possess certain \emph{separateness} property. This property also applies to labeled hypergraphs.

\begin{definition} \label{def.separatedLH}
A labeled hypergraph $H = (V,f)$ is \emph{separated} if for every pair of vertices $v,w \in V$, there exist edges $F,G \in \E$ such that $v \in F \setminus G$ and $w \in G \setminus F$.
\end{definition}

Clearly, since the vertices of $H_I$ correspond to minimal generators of $I$, $H_I$ is separated. Conversely, if $H = (V,f)$ is a separated labeled hypergraph, then we can define a squarefree monomial ideal $I_H$ in the polynomial ring $k[\A]$ by taking
$$I_H = \langle \prod_{v \in f(a)} a ~\big|~ v \in V \rangle.$$
It is easy to see that $H_{I_H}$ is the same as $H$ up to a permutation of the vertices. We summarize this property in the following proposition.

\begin{proposition} \label{prop.1-1}
There is a one-to-one correspondence
$$\left\{ \begin{array}{ll} \text{squarefree monomial} \\ \text{ideals in } k[x_1, \dots, x_n]\end{array} \right\} \longleftrightarrow \left\{\begin{array}{ll} \text{separated labeled hypergraphs} \\ \text{on } \A = \{x_1, \dots, x_n\} \\ \text{up to vertex permutation}\end{array} \right\},$$
with $I \mapsto H_I$ and $I_H \mapsfrom H$.
\end{proposition}

\begin{remark} Let $H$ be a hypergraph on $n$ vertices and $m$ edges. The \emph{incidence matrix} of $H$ is an $n \times m$ matrix whose $(i,j)$-entry is 1 if the $i$th vertex belongs to the $j$th edge and 0 otherwise. For a labeled hypergraph $H = H_I$, if we think of an edge with $t$ labels as $t$ distinct copies of that edge, then the incidence matrix of $H$ is exactly the transpose of the incidence matrix of the hypergraph whose edge ideal is $I$. This fact no longer holds if we restrict to unlabeled hypergraphs. This is another reason as why we choose to use the notion of labeled hypergraphs instead of those of unlabeled and dual hypergraphs.
\end{remark}

\noindent{\bf Notation.} Let $\P$ be a 0-1 polytope in $\RR^n$ and let $I_\P$ be its associated squarefree monomial ideal. Let $H_\P$ be the labeled hypergraph corresponding to $I_\P$. For simplicity of notation, we shall also denote the Ehrhart and polytopal rings of $\P$ by $\A[H_\P]$ and $k[H_\P]$, respectively. On the other hand, for a separated labeled hypergraph $H$, we shall denote by $\P_H$ the 0-1 polytope whose vertices correspond to the generators of the squarefree monomial ideal $I_H$. Proposition \ref{prop.1-1} allows us to move freely back and forth between 0-1 polytopes, squarefree monomial ideals and separated labeled hypergraphs. When using the notation $k[H]$, we often use the terminology \emph{toric ring} instead of \emph{polytopal ring}.


\section{Combinatorial criterion and simple reductions} \label{sec.red}

In this section, we recall a combinatorial criterion for the normality of 0-1 polytopes. We shall also provide a number of simple reductions that reduce the question to that of smaller or simpler hypergraphs.



For an interval $J \subseteq \RR$, let $\QQ_J$ denote the set of rational numbers in $J$. 

\begin{proposition} \label{thm.normalcondition}
Let $\P \subseteq \RR^n$ be a 0-1 polytope and assume that $E = \{\a_1, \dots, \a_s\}$ is the vertex set of $\P$. The polytope $\P$ is normal if and only if whenever $\a = \sum_{i=1}^s c_i\a_i \in \NN^n$ for $0 \le c_i < 1$ and $\sum_{i=1}^s c_i \in \ZZ$ then $\a$ can be rewritten as $\a = \sum_{i=1}^s d_i\a_i$ where $d_i \in \NN$ and $\sum_{i=1}^s d_i = \sum_{i=1}^s c_i$.
\end{proposition}

\begin{proof} Let $\FF$ represent $\NN, \QQ_{\ge 0}$ or $\QQ_{[0,1)}$, and denote by $(\FF E)_t$ the set $$\{\sum_{i=1}^s c_i\a_i ~|~ c_i \in \FF \text{ and } \sum_{i=1}^s c_i = t\}.$$
Since $E$ consists of all integral points in $\P$ it follows from the definition that $\P$ is normal if and only if $t\P \cap \ZZ^n = (\NN E)_t$ for all $t \in \NN$, i.e.,
\begin{align}
(\QQ_{\ge 0} E)_t \cap \ZZ^n = (\NN E)_t \ \forall \ t \in \NN. \label{eq.normal}
\end{align}

Clearly, if the equality in (\ref{eq.normal}) holds then $(\QQ_{[0,1)}E)_t \cap \ZZ^n \subseteq (\NN E)_t$ for all $t \in \NN$. On the other hand, if $(\QQ_{[0,1)}E)_t \cap \ZZ^n \subseteq (\NN E)_t$ for any $t \in \NN$ then by adding positive integer multiples of $\a_i$'s, the equality (\ref{eq.normal}) also holds.
\end{proof}

\begin{remark} \label{rmk.referee}
If $\P$ is normal, i.e., $\A[\P] = k[\P]$, then the polytopal ring $k[\P]$ is normal. The converse is not necessarily true; that is, the normality of the polytopal ring $k[\P]$ does not imply the normality of $\P$. In fact, the characterization in \cite[Proposition 3.5]{DV} states that $\A[\P] = k[\P]$ if and only if $k[\P]$ is normal and $\ZZ^{n+1}/\ZZ\{(\a_1, 1), \dots, (\a_s,1)\}$ is torsion-free.

The following particular example was provided to us by an anonymous referee. Consider the 0-1 polytope $\P$ in $\NN^7$ with vertices $\a_1, \dots, \a_6$ being the columns of the following matrix:
$$\left(\begin{array}{cccccc} 1 & 1 & 0 & 0 & 0 & 0 \\ 1 & 0 & 1 & 0 & 0 & 0 \\ 0 & 1 & 1 & 0 & 0 & 0 \\ 0 & 0 & 0 & 1 & 1 & 0 \\ 0 & 0 & 0 & 1 & 0 & 1 \\ 0 & 0 & 0 & 0 & 1 & 1 \\ 0 & 0 & 1 & 0 & 0 & 1 \end{array} \right).$$
The polytopal ring of $\P$ is $k[\P] = k[x_1x_2y, x_1x_3y, x_2x_3x_7y, x_4x_5y, x_4x_6y, x_5x_6x_7y]$ has dimension 6, so it is isomorphic to a polynomial ring in six variables. Thus, $k[\P]$ is normal. On the other hand, $\a = \sum_{i=1}^6 \dfrac{1}{2} \a_i = (1, \dots, 1) \in 3\P \cap \ZZ^7$, but $\a$ cannot be written as an integral linear combination of the points $\a_i$'s. Therefore, we can see directly that $\P$ is not a normal polytope.
\end{remark}

The following result allows us to remove all closed vertices from a labeled hypergraph in examining the equality $\A[H] = k[H]$.

\begin{proposition} \label{prop.removingclosevertices}
Let $H$ be a labeled hypergraph, and let $H'$ be the labeled hypergraph obtained from $H$ by removing all closed vertices (and contracting edges containing these vertices) from $H$. Then $\A[H] = k[H]$ if and only if $\A[H'] = k[H']$.
\end{proposition}

\begin{proof} Let $(x^{\a_1}, \dots, x^{\a_s}) = I_H$ be the squarefree monomial ideal of $H$. Assume that $x^{\a_s}$ corresponds to a closed vertex $s$ in $H$ and let $v$ be the label of $s$ that does not belong to any other edge. Take any $\a = \sum_{i=1}^s c_i \a_i \in \NN^n$ for $0 \le c_i < 1$. Observe that the power of $v$ in $x^\a$ is exactly $c_s$. Thus, for $\a$ to be an integral point, we must have $c_s = 0$. In other words, in applying Proposition \ref{thm.normalcondition} to $\P_H$, we can omit the vertex $s$ in $H$. This is true for any closed vertex in $H$. Therefore, the criterion of Proposition \ref{thm.normalcondition} is equivalent when applying to $\P_H$ and $\P_{H'}$.
\end{proof}

\begin{example} Consider the hypergraph $H$ in Figure \ref{fig.RemoveClose}. By successively removing closed vertices, $H$ is reduced to $H''$ (as in Figure \ref{fig.Removed}) which does not contain any closed vertices. Proposition \ref{prop.removingclosevertices} says that $\A[H] = k[H]$ if and only if $\A[H''] = k[H'']$. 

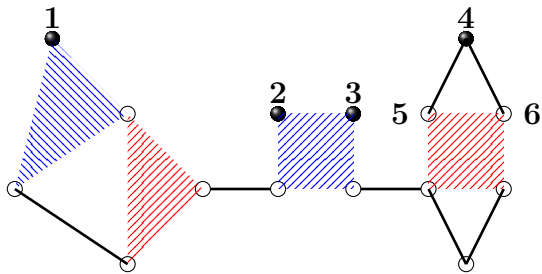
\begin{figure}[h]
\begin{center}
\begin{tikzpicture}

\draw  [shape=circle] (-3.5,0) circle (.1);
\draw  [shape=circle] (-2,1) circle (.1);
\draw  [shape=circle] (-2,-1) circle (.1);
\draw  [shape=circle] (-1,0) circle (.1);
\draw  [shape=circle] (0,0) circle (.1);
\draw  [shape=circle] (2,1) circle (.1) node [left] {$\bf{5}$};
\draw  [shape=circle] (3,1) circle (.1)node [right] {$\bf{6}$};
\draw  [shape=circle] (2,0) circle (.1);
\draw  [shape=circle] (3,0) circle (.1);
\draw  [shape=circle] (2.5,-1) circle (.1);
\draw  [shape=circle] (1,0) circle (.1);
\shade [shading=ball, ball color=black]  (-3,2) circle (.1) node [above] {$\bf{1}$};
\shade [shading=ball, ball color=black]  (0,1) circle (.1) node [above] {$\bf{2}$};
\shade [shading=ball, ball color=black]  (1,1) circle (.1) node [above] {$\bf{3}$};
\shade [shading=ball, ball color=black]  (2.5,2) circle (.1) node [above] {$\bf{4}$};

\draw [line width=1pt ] (-3.5,0)--(-2,-1)  ;
\draw [line width=1pt ] (-0.9,0)--(-0.1,0)  ;
\draw [line width=1pt ] (1.1,0)--(1.9,0)  ;
\draw [line width=1pt ] (2,1)--(2.5,2)  ;
\draw [line width=1pt ] (2.5,2)--(3,1)  ;
\draw [line width=1pt ] (2,0)--(2.5,-1)  ;
\draw [line width=1pt ] (2.5,-1)--(3,0)  ;

\path [pattern=north west lines, pattern  color=blue]   (-3,2)--(-3.5,0)--(-2,1)--cycle;
\path [pattern=north east lines, pattern color=red]     (-2,1)--(-2,-1)--(-1,0)--cycle;
\path [pattern=north east lines, pattern color=blue]     (0,1)--(1,1)--(1,0)--(0,0)--cycle;
\path [pattern=north east lines, pattern color=red]     (2,0)--(3,0)--(3,1)--(2,1)--cycle;

\end{tikzpicture}
\end{center}
\caption{A labeled hypergraph $H$.}\label{fig.RemoveClose}
\end{figure}

\begin{figure}[h]
\begin{center}
\begin{tikzpicture}

\draw  [shape=circle] (-3.5,-3) circle (.1);
\draw  [shape=circle] (-2,-2) circle (.1);
\draw  [shape=circle] (-2,-4) circle (.1);
\draw  [shape=circle] (-1,-3) circle (.1);
\draw  [shape=circle] (0,-3) circle (.1);
\draw  [shape=circle] (1,-3) circle (.1);
\shade [shading=ball, ball color=black] (2,-2) circle (.1) node [left] {$\bf{5}$};
\shade [shading=ball, ball color=black]  (3,-2) circle (.1) node [right] {$\bf{6}$};
\draw  [shape=circle] (2,-3) circle (.1);
\draw  [shape=circle] (3,-3) circle (.1);
\draw  [shape=circle] (2.5,-4) circle (.1);
\draw [line width=1pt ] (-3.5,-3)--(-2,-4)  ;
\draw [line width=1pt ] (-3.5,-3)--(-2,-2)  ;
\draw [line width=1pt ] (0.1,-3)--(0.9,-3)  ;
\draw [line width=1pt ] (-0.9,-3)--(-0.1,-3)  ;
\draw [line width=1pt ] (1.1,-3)--(1.9,-3)  ;
\draw [line width=1pt ] (2,-3)--(2.5,-4)  ;
\draw [line width=1pt ] (2.5,-4)--(3,-3)  ;

\path [pattern=north east lines, pattern color=red]     (-2,-2)--(-2,-4)--(-1,-3)--cycle;
\path [pattern=north east lines, pattern color=red]     (2,-3)--(3,-3)--(3,-2)--(2,-2)--cycle;

\end{tikzpicture}
\end{center}
\caption{$H'$ obtained by removing vertices $1,2,$ and $4$ from $H$.}
\end{figure}
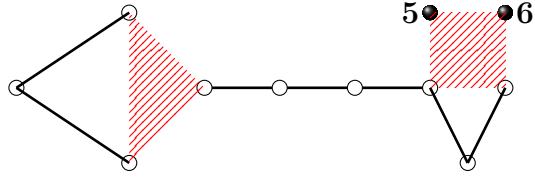

\begin{figure}[h]
\begin{center}
\begin{tikzpicture}

\draw  [shape=circle] (-3.5,-6) circle (.1);
\draw  [shape=circle] (-2,-5) circle (.1);
\draw  [shape=circle] (-2,-7) circle (.1);
\draw  [shape=circle] (-1,-6) circle (.1);
\draw  [shape=circle] (0,-6) circle (.1);
\draw  [shape=circle] (1,-6) circle (.1);
\draw  [shape=circle] (2,-6) circle (.1);
\draw  [shape=circle] (3,-6) circle (.1);
\draw  [shape=circle] (2.5,-7) circle (.1);
\draw [line width=1pt ] (-3.5,-6)--(-2,-7)  ;
\draw [line width=1pt ] (-3.5,-6)--(-2,-5)  ;
\draw [line width=1pt ] (0.1,-6)--(0.9,-6)  ;
\draw [line width=1pt ] (-0.9,-6)--(-0.1,-6)  ;
\draw [line width=1pt ] (1.1,-6)--(1.9,-6)  ;
\draw [line width=1pt ] (2,-6)--(2.5,-7)  ;
\draw [line width=1pt ] (2.5,-7)--(3,-6)  ;
\draw [line width=1pt ] (2.1,-6)--(2.9,-6)  ;
\path [pattern=north east lines, pattern color=red]     (-2,-5)--(-2,-7)--(-1,-6)--cycle;

\end{tikzpicture}
\end{center}
\caption{$H''$ obtained by removing vertices $5$ and $6$ from $H'$.} \label{fig.Removed}
\end{figure}
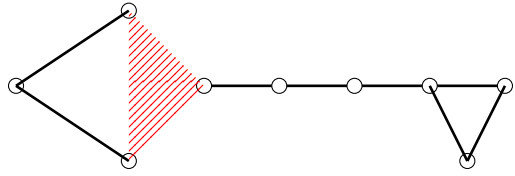
\end{example}

\begin{remark} Removing closed vertices from $H$ has the same effect as removing prime elements from the affine semigroup of integral points in the cone of $\P_H$.
\end{remark}

The next observation inspires us to focus only to labeled hypergraphs which do contain special odd cycles. Note that hypergraphs without special odd cycles are called \emph{balanced} hypergraphs; their incidence matrices are also called \emph{balanced} matrices. Balanced hypergraphs, matrices and simplicial complexes are well studied from both algorithmic and theoretical perspectives (cf. \cite{S}).

\begin{proposition} \label{thm.nooddcycle}
Let $H = H_I$ be a balanced labeled hypergraph and assuming that the minimal generators of $I$ are of the same degree. Then $\A[H] = k[H]$.
\end{proposition}

\begin{proof} Let $\Delta = \Delta(I)$ be the simplicial complex whose facets correspond to generators of $I$. By looking at the incidence matrices and their transposes, since $H$ is balanced, we have that $\Delta$ is also balanced. It follows from \cite[Theorem 2.5]{HHTZ} (see also \cite[Theorem 4.6 and Proposition 4.10]{GRV}) that $I^{(t)} = I^t$ for all $t \in \NN$. Since $I^t \subseteq \overline{I^t} \subseteq I^{(t)}$, we now have $I^t = \overline{I^t}$ for all $t \in \NN$.

Let $\P = \P_H$ be the 0-1 polytope associated to $I$ and let $E = \{\a_1, \dots, \a_s\}$ be the vertex set of $\P$. Clearly, any $\a \in (\QQ_{\ge 0}E)_t \cap \ZZ^n$ would give rise to an integral equation for $x^\a$, and so $x^\a \in \overline{I^t} = I^t$. That is, $x^\a = x^\delta (x^{\a_1})^{d_1} \dots (x^{\a_s})^{d_s}$, where $\delta \in \NN^n$, $d_i \in \NN$ and $\sum_{i=1}^s d_i = t$. Since the minimal generators of $I$ are of the same degree, i.e., $|\a_1| = \dots = |\a_s|$, this and the fact that $\a \in (\QQ_{\ge 0}E)_t$ imply that $\delta = 0$. Thus, $(\QQ_{\ge 0}E)_t \cap \ZZ^n = (\NN E)_t$ for any $t \in \NN$, and the result is proved by Proposition \ref{thm.normalcondition}.
\end{proof}

\begin{remark} A shorter argument to Proposition \ref{thm.nooddcycle} can be obtained by making use of \cite[Corollary 1.7]{HHTZ} or \cite[Proposition 4.4]{GRV} after deducing that $I^{(t)} = I^t$ for all $t \in \NN$.
\end{remark}

\begin{example} The condition that $I$ is generated in a single degree in Proposition \ref{thm.nooddcycle} is important. The following example, given in \cite[Example 3.10]{MOV}, illustrates that. Let
$$I = (x_1x_2x_3x_4,x_5x_6x_7x_8,x_1x_5,x_2x_6,x_3x_7,x_4x_8) \subseteq k[x_1, \dots, x_8].$$
Then the labeled hypergraph $H = H_I$ is depicted in Figure \ref{fig.refExample}. Since $H$ is a bipartite graph, $H$ contains no odd cycles, i.e., $H$ is balanced. On the other hand, a Macaulay 2 computation shows that $\A[H] \not= k[H]$.

\begin{figure}[h]
\begin{center}
\begin{tikzpicture}
\draw  [shape=circle] (2,1) circle (.1) node [left] {$H_I$};
\draw  [shape=circle] (3,1) circle (.1) ;
\draw  [shape=circle] (1,0) circle (.1) ;
\draw  [shape=circle] (2,0) circle (.1) ;
\draw  [shape=circle] (3,0) circle (.1) ;
\draw  [shape=circle] (4,0) circle (.1) ;

\draw [line width=1.2pt ] (2,1)--(1,0)  ;
\draw [line width=1.2pt ] (2,1)--(2,0)  ;
\draw [line width=1.2pt ] (2,1)--(3,0)  ;
\draw [line width=1.2pt ] (2,1)--(4,0)  ;
\draw [line width=1.2pt ] (3,1)--(1,0)  ;
\draw [line width=1.2pt ] (3,1)--(2,0)  ;
\draw [line width=1.2pt ] (3,1)--(3,0)  ;
\draw [line width=1.2pt ] (3,1)--(4,0)  ;

\end{tikzpicture}
\end{center}
\caption{A balanced hypergraph $H$ with $\A[H] \not= k[H]$.} \label{fig.refExample}
\end{figure}
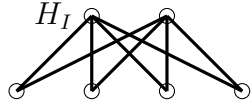
\end{example}

In investigating when $\A[H] \not= k[H]$, it turns out that the strict containment $k[H] \varsubsetneq \A[H]$ for a smaller hypergraph implies that of larger hypergraphs. We prove this observation in the following theorem.

\begin{theorem} \label{thm.subhypergraph}
Let $H$ be a labeled hypergraph and let $H'$ be a minor of $H$. If $\A[H'] \not= k[H']$, then $\A[H] \not= k[H]$.
\end{theorem}

\begin{proof} Suppose that
$\{x^{\a_1}, \dots, x^{\a_s}\} \text{ and } \{x^{\a_1}, \dots, x^{\a_s}, x^{\a_{s+1}}, \dots, x^{\a_t}\}$
are the minimal generators for $I_{H'}$ and $I_H$, respectively. Since $\A[H'] \not= k[H']$, by Proposition \ref{thm.normalcondition}, there exists
$$\a = \sum_{i=1}^s c_i\a_i \in \NN^n \text{ with } 0 < c_i < 1 \text{ and } \sum_{i=1}^s c_i \in \ZZ$$
such that $\a$ cannot be written as a nonnegative {\it integral} linear combination of the $\a_i$'s (with the same sum of coefficients). Suppose, by contradiction, that $\A[H] = k[H]$. Consider $\a = \sum_{j=1}^s c_j\a_j + \sum_{j=s+1}^t 0\cdot\a_j$. By Proposition \ref{thm.normalcondition}, we can write $\a$ as
$$\a = \sum_{i=1}^t d_i \a_i, \ d_i \in \NN \ \forall i \text{ and } \sum_{i=1}^t d_i = \sum_{j=1}^s c_j.$$

Let $z$ be any variable appearing in some of $\{\a_{s+1}, \dots, \a_t\}$ but not appearing in any of $\{\a_1, \dots, \a_s\}$. Clearly, the power of $z$ in $x^\a$ must be 0. Thus, $d_j = 0$ for any $j$ such that $z \big| x^{\a_j}$. Now, by the construction of induced subhypergraph, each generator $\{x^{\a_{s+1}}, \dots, x^{\a_t}\}$ belongs to some edges that were removed from $H$ to obtained $H'$. Therefore, in each generators $x^{\a_j}$, for $j = s+1, \dots, t$, there must be a variable $z$ that does not divide any of the generators of $I_{H'}$ (take a variable belonging to the label of a maximal face containing $x^{\a_j}$). Hence, $d_j = 0$ for all $j=s+1, \dots, t$, and we have
$$\a = \sum_{i=1}^s d_i\a_i, \ d_i \in \NN \ \forall i \text{ and } \sum_{i=1}^s d_i = \sum_{j=1}^s c_j.$$
This is a contradiction to assumption that $\A[H'] \not= k[H']$, and thus, $\A[H] \not= k[H]$.
\end{proof}

\begin{remark} Considering minors of $H$ has the same effect as restricting to faces of the polytope $\P_H$ on coordinate hyperplanes.
\end{remark}

\begin{figure}[h]
\begin{center}
\begin{tikzpicture}

\draw  [shape=circle] (-3,0) circle (.1);
\draw  [shape=circle] (-2.5,.5) circle (.1);
\draw  [shape=circle] (-1.5,.5) circle (.1);
\draw  [shape=circle] (-2.5,-.5) circle (.1);
\draw  [shape=circle] (-1.5,-.5) circle (.1);
\draw  [shape=circle] (-1,0) circle (.1) ;
\draw  [shape=circle] (0,0) circle (.1);
\draw  [shape=circle] (1,0) circle (.1);
\draw  [shape=circle] (2,0) circle (.1);
\draw  [shape=circle] (1,-1) circle (.1);
\draw  [shape=circle] (2,-1) circle (.1);
\draw  [shape=circle] (2.5,-.5) circle (.1);
\draw  [shape=circle] (-1,-1.5) circle (.1);
\draw  [shape=circle] (0,-1.5) circle (.1);
\draw  [shape=circle] (-1,-2.5) circle (.1);
\draw  [shape=circle] (0,-2.5) circle (.1);

\draw [line width=1pt ] (-2.4,.5)--(-1.6,.5)  ;
\draw [line width=1pt ] (-3,0)--(-2.5,.5)  ;
\draw [line width=1pt ] (-3,0)--(-2.5,-.5)  ;
\draw [line width=1pt ] (-0.9,0)--(-0.1,0)  ;
\draw [line width=1pt ] (0.1,0)--(0.9,0)  ;
\draw [line width=1pt ] (2,0)--(2.5,-.5)  ;
\draw [line width=1pt ] (2,-1)--(2.5,-.5)  ;

\draw [line width=1pt ] (-0.9,-1.5)--(-0.1,-1.5)  node [pos=.5, above] {$E_{1}$};
\draw [line width=1pt ] (-0.9,-2.5)--(-0.1,-2.5)  node [pos=.5, above] {$E_{3}$};
\draw [line width=1pt ] (-1,-1.6)--(-1,-2.4)  node [pos=.5, left] {$E_{2}$};
\draw [line width=1pt ] (0,-1.6)--(0,-2.4)  node [pos=.5, right] {$E_{4}$};
\draw [line width=1pt ] (0,-1.5)--(1,-1)  node [pos=.5, above] {$E_{5}$};

\path [pattern=north west lines, pattern  color=blue]   (-1.5,.5)--(-1,0)--(-1.5,-.5)--cycle;
\path [pattern=north east lines, pattern color=red]     (-1.5,-.5)--(-1,-1.5)--(-2.5,-.5)--cycle;
\path [pattern=north east lines, pattern color=blue]     (1,0)--(2,0)--(2,-1)--(1,-1)--cycle;

\end{tikzpicture}
\end{center}
\caption{A labeled hypergraph $H$.} \label{fig.b4minor}
\end{figure}
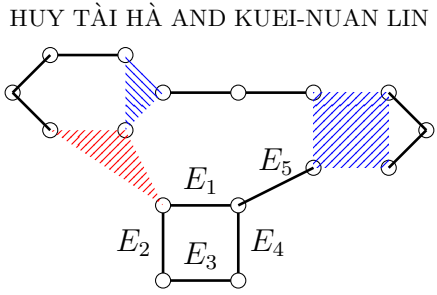

\begin{figure}[h]
\begin{center}
\begin{tikzpicture}

\draw  [shape=circle] (-3,-4) circle (.1);
\draw  [shape=circle] (-2.5,-3.5) circle (.1);
\draw  [shape=circle] (-1.5,-3.5) circle (.1);
\draw  [shape=circle] (-2.5,-4.5) circle (.1);
\draw  [shape=circle] (-1.5,-4.5) circle (.1);
\draw  [shape=circle] (-1,-4) circle (.1) ;
\draw  [shape=circle] (0,-4) circle (.1);
\draw  [shape=circle] (2,-4) circle (.1);
\draw  [shape=circle] (2,-5) circle (.1);
\draw  [shape=circle] (1,-4) circle (.1);
\draw  [shape=circle] (2.5,-4.5) circle (.1);
\draw [line width=1pt ] (-2.4,-3.5)--(-1.6,-3.5)  ;
\draw [line width=1pt ] (-3,-4)--(-2.5,-3.5)  ;
\draw [line width=1pt ] (-3,-4)--(-2.5,-4.5)  ;
\draw [line width=1pt ] (-0.9,-4)--(-0.1,-4)  ;
\draw [line width=1pt ] (0.1,-4)--(0.9,-4)  ;
\draw [line width=1pt ] (2,-4)--(2.5,-4.5)  ;
\draw [line width=1pt ] (2,-5)--(2.5,-4.5)  ;
\draw [line width=1pt ] (-2.4,-4.5)--(-1.6,-4.5)  ;
\path [pattern=north west lines, pattern  color=blue]   (-1.5,-3.5)--(-1,0-4)--(-1.5,-4.5)--cycle;
\path [pattern=north east lines, pattern color=blue]     (1,-4)--(2,-4)--(2,-5)--cycle;

\end{tikzpicture}
\end{center}
\caption{A minor $H'$ of $H$ in Figure \ref{fig.b4minor}.} \label{fig.minor}
\end{figure}
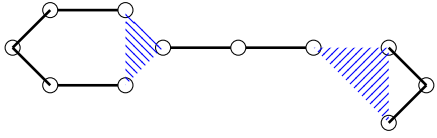

\begin{example} By removing edges $E_1, \dots, E_5$ and vertices belonging to these edges in the labeled hypergraph $H$ of Figure \ref{fig.b4minor} we obtain a minor $H'$ as in Figure \ref{fig.minor}. As we shall see from Definition \ref{def.exceptional} and Theorem \ref{thm.exceptional}, $\A[H'] \not= k[H']$. Thus, it follows from Theorem \ref{thm.subhypergraph} that $\A[H] \not= k[H]$.
\end{example}


\section{Ehrhart and toric rings of labeled hypergraphs} \label{sec.notnormal}

In this section, we give a necessary and sufficient condition for the equality $\A[H] = k[H]$ in the case where the 1-skeleton of $H$ is connected and contains odd cycles. In particular, we identify a large class of labeled hypergraphs for which the toric ring $k[H]$ is normal. We shall also provide sufficient conditions for labeled hypergraphs so that $\A[H] \not= k[H]$ (or equivalently, necessary conditions for $\A[H] = k[H]$).

We start by considering the case where the 1-skeleton of $H$ is connected and contains odd cycles.

\begin{theorem} \label{thm.oneoddcycle}
Let $H$ be a separated labeled hypergraph. Assume that the 1-skeleton of $H$ is connected and contains odd cycles. Then $\A[H] = k[H]$ if and only if either of the following conditions is satisfied:
\begin{enumerate}
\item $H$ has an odd number of vertices;
\item $H$ contains an even dimensional edge.
\end{enumerate}
\end{theorem}

\begin{proof} Let $I_H = (x^{\a_1}, \dots, x^{\a_s})$. Consider an integral point $\a = \sum_{i=1}^s c_i \a_i \in \NN^n$ for $0 < c_i < 1$ and $\sum_{i=1}^s c_i \in \ZZ$. Let $l$ and $m$ be any two adjacent vertices in the 1-skeleton of $H$, and assume that $v$ is a label for the edge $\{l,m\}$. It can be seen that the power of $v$ in $x^\a$ is $c_l + c_m$. Thus, for $\a$ to be an integral point, we must have $c_l + c_m = 1$. Since the 1-skeleton of $H$ is connected and contains an odd cycle, by tracing around an odd cycle, we can conclude that $c_i = 1/2$ for all $i$.

If $H$ has an odd number of vertices (i.e., $s$ is an odd number), then $\sum_{i=1}^s c_i = s/2 \not\in \ZZ$, a contradiction. On the other hand, if $H$ contains an even dimensional edge $E$ and $w$ is a label of $E$ then the power of $w$ in $x^\a$ is $\sum_{j \in E}c_j = |E|/2 \not\in \ZZ$, also a contradiction. Hence, if either (1) or (2) is satisfied then no such integral point $\a$ exists, and so $\A[H] = k[H]$ by Proposition \ref{thm.normalcondition}.

Conversely, suppose that both (1) and (2) fails, i.e., $H$ has an even number of vertices and contains no even dimensional edges. Consider $\a = \sum_{i=1}^s \dfrac{1}{2}\a_i$. Since $s$ is even, $\sum_{i=1}^s \dfrac{1}{2} \in \ZZ$. Also, for any edge $E$ in $H$ and any label $w$ of $E$, the power of $w$ in $x^\a$ is $\sum_{j \in E} \dfrac{1}{2} = \dfrac{|E|}{2} \in \ZZ$. Thus, $\a \in \NN^n$. Now, assume that we can write $\a = \sum_{i=1}^s d_i\a_i$, where $d_i \in \NN$ (and $\sum_{i=1}^s d_i = \dfrac{s}{2}$). Let $C = (i_1, \dots, i_{2l+1})$ be an odd cycle in the 1-skeleton of $H$. Then since the power of any label of each edge of $C$ in $x^\a$ is 1, we must have $d_{i_j} = 0,$ or $1$. Since $C$ has an odd number of vertices, there must exist a $j$ such that $d_{i_j} = d_{i_{j+1}} = 1$. Take $u$ to be any label of the edge $\{i_j, i_{j+1}\}$ in $C$, then the power of $u$ in $x^\a$ is now $d_{i_j} + d_{i_{j+1}} = 2$, a contradiction. Hence, we cannot write $\a = \sum_{i=1}^s d_i\a_i$ for $d_i \in \NN$. By Proposition \ref{thm.normalcondition}, this implies that $\A[H] \not= k[H]$.
\end{proof}

\begin{corollary} Let $H$ be a separated labeled hypergraph whose 1-skeleton is connected and contains odd cycles. Assume that $H$ either has an odd number of vertices or contains an even dimensional edge. Then $k[H]$ is normal.
\end{corollary}

\begin{proof} The result follows immediately from Theorem \ref{thm.oneoddcycle} noticing that the Ehrhart ring $\A[H]$ is always normal.
\end{proof}

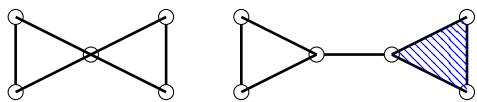
\begin{figure}[h]
\begin{center}
\begin{tikzpicture}

\draw  [shape=circle] (-1,0.5) circle (.1); 
\draw  [shape=circle] (-1,-0.5) circle (.1);
\draw  [shape=circle] (0,0) circle (.1);
\draw  [shape=circle] (1,-.5) circle (.1);
\draw  [shape=circle] (1,.5) circle (.1);

\draw [line width=1pt ] (-1,.4)--(-1,-0.4)  ;
\draw [line width=1pt ] (-1,0.5)--(0,0)  ;
\draw [line width=1pt ] (-1,-0.5)--(0,0)  ;
\draw [line width=1pt ] (1,0.5)--(0,0)  ;
\draw [line width=1pt ] (1,-0.5)--(0,0)  ;
\draw [line width=1pt ] (1,-0.4)--(1,0.4)  ;

\draw  [shape=circle] (2,0.5) circle (.1); 
\draw  [shape=circle] (2,-0.5) circle (.1);
\draw  [shape=circle] (3,0) circle (.1);
\draw  [shape=circle] (4,0) circle (.1);
\draw  [shape=circle] (5,-.5) circle (.1);
\draw  [shape=circle] (5,0.5) circle (.1);

\draw [line width=1pt ] (2,.4)--(2,-0.4)  ;
\draw [line width=1pt ] (2,0.5)--(3,0)  ;
\draw [line width=1pt ] (2,-0.5)--(3,0)  ;
\draw [line width=1pt ] (3.9,0)--(3.1,0)  ;
\draw [line width=1pt ] (5,0.5)--(4,0)  ;
\draw [line width=1pt ] (5,-0.5)--(4,0)  ;
\draw [line width=1pt ] (5,-0.4)--(5,0.4)  ;

\path [pattern=north west lines, pattern  color=blue]   (4,0)--(5,-0.5)--(5,0.5)--cycle;
\end{tikzpicture}
\end{center}
\caption{Hypergraphs with connected 1-skeleton that contains odd cycles where $k[H] = \A[H]$ is normal.} \label{fig.1skeleton}
\end{figure}

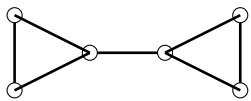
\begin{figure}[h]
\begin{center}
\begin{tikzpicture}

\draw  [shape=circle] (2,-1.5) circle (.1); 
\draw  [shape=circle] (2,-2.5) circle (.1);
\draw  [shape=circle] (3,-2) circle (.1);
\draw  [shape=circle] (4,-2) circle (.1);
\draw  [shape=circle] (5,-2.5) circle (.1);
\draw  [shape=circle] (5,-1.5) circle (.1);

\draw [line width=1pt ] (2,-1.6)--(2,-2.4)  ;
\draw [line width=1pt ] (2,-1.5)--(3,-2)  ;
\draw [line width=1pt ] (2,-2.5)--(3,-2)  ;
\draw [line width=1pt ] (3.9,-2)--(3.1,-2)  ;
\draw [line width=1pt ] (5,-1.5)--(4,-2)  ;
\draw [line width=1pt ] (5,-2.5)--(4,-2)  ;
\draw [line width=1pt ] (5,-2.4)--(5,-1.6)  ;


\end{tikzpicture}
\end{center}
\caption{A hypergraph with connected 1-skeleton that contains odd cycles where $\A[H] \not= k[H]$.}\label{fig.1skeletonN}
\end{figure}

\begin{example} Hypergraphs in Figures \ref{fig.1skeleton} and \ref{fig.1skeletonN} all have connected 1-skeletons which contain odd cycles. By Theorem \ref{thm.oneoddcycle}, the toric rings of hypergraphs in Figure \ref{fig.1skeleton} are normal and the toric ring of the hypergraph in Figure \ref{fig.1skeletonN} is not normal.
\end{example}

Our next result, in the case where the 1-skeleton of $H$ is connected (but does not necessarily contain odd cycles), provides a general sufficient condition for labeled hypergraphs for which $\A[H] \not= k[H]$.

\begin{definition} \label{def.coloring}
A labeled hypergraph $H$ is said to be \emph{2-solvable modulo $p$}, for a prime number $p$, if there exists a 2-coloring of the 1-skeleton of $H$ (using two colors \emph{red} and \emph{blue}) with the following property: for any edge $E$ in $H$ (and for $E = V$) by letting $r_E$ and $b_E$ be the number of red and blue vertices in $E$ (or the total number of red and blue vertices in $H$ when $E = V$) we have $r_E - b_E = 0$ in $\ZZ_p$. The hypergraph $H$ is said to be \emph{2-solvable} if it is 2-solvable modulo $p$ for some prime number $p$.
\end{definition}

\begin{example} The hypergraph in Figure \ref{fig.bicolor} is 2-solvable modulo 3, and a 2-coloring of its 1-skeleton is given.

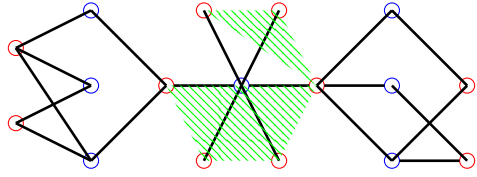
\begin{figure}[h]
\begin{center}
\begin{tikzpicture}

\draw  [shape=circle, color=red] (-3,-0.5) circle (.1);
\draw  [shape=circle, color=red] (-3,0.5) circle (.1);
\draw  [shape=circle, color=blue] (-2,1) circle (.1);
\draw  [shape=circle, color=blue] (-2,0) circle (.1);
\draw  [shape=circle, color=blue] (-2,-1) circle (.1);
\draw  [shape=circle, color=red] (-1,0) circle (.1);
\draw  [shape=circle, color=red] (-0.5,-1) circle (.1);
\draw  [shape=circle, color=red] (-0.5,1) circle (.1);
\draw  [shape=circle, color=blue] (0,0) circle (.1);
\draw  [shape=circle, color=red] (0.5,-1) circle (.1);
\draw  [shape=circle, color=red] (0.5,1) circle (.1);
\draw  [shape=circle, color=red] (1,0) circle (.1);
\draw  [shape=circle, color=blue] (2,0) circle (.1);
\draw  [shape=circle, color=blue] (2,1) circle (.1);
\draw  [shape=circle, color=blue] (2,-1) circle (.1);
\draw  [shape=circle, color=red] (3,0) circle (.1);
\draw  [shape=circle, color=red] (3,-1) circle (.1);

\draw [line width=1pt ] (-3,.5)--(-2,-1)  ;
\draw [line width=1pt ] (-3,0.5)--(-2,0)  ;
\draw [line width=1pt ] (-3,0.5)--(-2,1)  ;
\draw [line width=1pt ] (-3,-0.5)--(-2,0)  ;
\draw [line width=1pt ] (-3,-0.5)--(-2,-1)  ;
\draw [line width=1pt ] (-2,1)--(-1,0)  ;
\draw [line width=1pt ] (-2,-1)--(-1,0)  ;
\draw [line width=1pt ] (-0.9,0)--(-0.1,0)  ;
\draw [line width=1pt ] (-0.5,1)--(0,0)  ;
\draw [line width=1pt ] (-0.5,-1)--(0,0)  ;
\draw [line width=1pt ] (0.5,1)--(0,0)  ;
\draw [line width=1pt ] (0.5,-1)--(0,0)  ;
\draw [line width=1pt ] (0.1,0)--(0.9,0)  ;
\draw [line width=1pt ] (1.9,0)--(1.1,0)  ;
\draw [line width=1pt ] (2,1)--(1,0)  ;
\draw [line width=1pt ] (2,-1)--(1,0)  ;
\draw [line width=1pt ] (2.9,-1)--(2.1,-1)  ;
\draw [line width=1pt ] (3,0)--(2,1)  ;
\draw [line width=1pt ] (3,0)--(2,-1)  ;
\draw [line width=1pt ] (3,-1)--(2,0)  ;

\path [pattern=north west lines, pattern  color=green]   (-0.5,1)--(0.5,1)--(1,0)--cycle;
\path [pattern=north west lines, pattern  color=green]   (-1,0)--(1,0)--(0.5,-1)--(-0.5,-1)--cycle;

\end{tikzpicture}
\end{center}
\caption{A hypergraph 2-solvable modulo 3.}\label{fig.bicolor}
\end{figure}
\end{example}

\begin{theorem} \label{thm.bi-color}
Let $H$ be a separated labeled hypergraph with connected 1-skeleton. Suppose that $H$ is 2-solvable and contains a simple face in which the number of red and blue vertices are different. Then $\A[H] \not= k[H]$.
\end{theorem}

\begin{proof}
As before, let $I_H = (x^{\a_1}, \dots, x^{\a_s})$. Let $p$ be a prime number such that $H$ is 2-solvable modulo $p$. Let $r$ and $b$ be the number of red and blue vertices in $H$. Notice that since the 1-skeleton of $H$ forms a connected graph, its 2-coloring is uniquely determined by the color of any vertex.

Let
$$c_i = \left\{ \begin{array}{lll} \dfrac{1}{p} & \text{if} & i \text{ is red} \\ \dfrac{p-1}{p} & \text{if} & i \text{ is blue} \end{array} \right.$$
and let $\a = \sum_{i=1}^s c_i \a_i.$ Since $p \big| r-b$,
$$\sum_{i=1}^s c_i = \dfrac{r}{p} + \dfrac{(p-1)b}{p} = \dfrac{r-b}{p} + b \in \ZZ.$$
Also, for any variable $v$, let $E$ be the edge having $v$ as a label, and let $r_E$ and $b_E$ be the number of red and blue vertices in $E$. Then, the power of $v$ in $x^{\a}$ is
$$\dfrac{r_E}{p} + \dfrac{(p-1)b_E}{p} = \dfrac{r_E - b_E}{p} + b_E \in \NN.$$
Thus, $\a \in \NN^n$.

To prove $\A[H] \not= k[H]$, in light of Proposition \ref{thm.normalcondition}, we will show that $\a$ cannot be written as a nonnegative integral combination of the $\a_i$'s. Suppose, by contradiction, that $\a = \sum_{i=1}^s d_i \a_i$, where $d_i \in \NN$ for all $i$. Observe that for any edge $F = \{l,m\}$ in the 1-skeleton of $H$, the colors of $l$ and $m$ are different. Thus, for any label $w$ of $F$, the power of $w$ in $x^\a$ is exactly $\dfrac{1}{p} + \dfrac{p-1}{p} = 1$. Hence, the coefficients $d_l$ and $d_m$ must be 0 and 1 (or 1 and 0). This, together with the fact that the 2-coloring in $H$ is uniquely determined by the color of any vertex, implies that among the new coefficients $d_i$'s either red vertices have coefficients 1 and blue vertices have coefficients 0, or vice-versa.

Let $G$ be a simple edge in $H$ such that its number of red and blue vertices ($r'$ and $b'$, respectively) are different. Consider any label $u$ of $G$. It follows from our observation above that the power of $u$ in $x^{\a}$ is either $r'$ or $b'$. However, by the definition of $\a$, the power of $u$ is $$\dfrac{r'}{p} + \dfrac{(p-1)b'}{p} = \dfrac{r'-b'}{p} + b',$$
which is different from both $r'$ and $b'$. Thus, we arrive at a contradiction.
\end{proof}

In the case where the 1-skeleton of $H$ is not necessarily connected, our next result is based on Theorem \ref{thm.subhypergraph}. This theorem allows us to seek for ``minimal structure'' that obstruct the equality between Ehrhart and toric rings. Inspired by the characterization for graphs of \cite{HO, SVV}, we consider a configuration of 2 odd cycles in our labeled hypergraphs. Our next result identifies such a configuration that prevents the equality $\A[H] = k[H]$.

\begin{definition} \label{def.exceptional}
Let $H$ be a separated labeled hypergraph. A pair of two odd cycles $C_1$ and $C_2$ in $H$ is called an \emph{exceptional} pair if the following conditions are satisfied:
\begin{enumerate}
\item $C_1$ and $C_2$ do not share any vertices nor edges;
\item for $i=1,2$, all but one edges of $C_i$ are 1-dimensional and the remaining edge is a simple edge;
\item the two (possibly) higher dimensional edges of $C_1$ and $C_2$ are connected by at least an edge that does not contain any vertices of $C_1$ and $C_2$;
\item $H$ does not have any edge that contains an odd number of vertices from $C_1$ and $C_2$.
\end{enumerate}
\end{definition}

\begin{example} In Figure \ref{fig.bowtie} we have a bow-tie in the sense of \cite{SVV} and the labeled hypergraph corresponding to its edge ideal. This labeled hypergraph has a disconnected 1-skeleton, but it consists of an exceptional pair of odd cycles.
\begin{figure}[h]
\begin{center}
\begin{tikzpicture}

\shade [shading=ball, ball color=black]  (1,0.5) circle (.1); 
\shade [shading=ball, ball color=black]  (1,-0.5) circle (.1);
\shade [shading=ball, ball color=black]   (2,0) circle (.1);
\shade [shading=ball, ball color=black]   (3,0) circle (.1);
\shade [shading=ball, ball color=black]  (4,0) circle (.1);
\shade [shading=ball, ball color=black]   (5,-.5) circle (.1);
\shade [shading=ball, ball color=black]   (5,0.5) circle (.1);

\draw [line width=1pt ] (1,.5)--(1,-0.5)  ;
\draw [line width=1pt ] (1,0.5)--(2,0)  ;
\draw [line width=1pt ] (1,-0.5)--(2,0)  ;
\draw [line width=1pt ] (3,0)--(2,0)  ;
\draw [line width=1pt ] (3,0)--(4,0)  ;
\draw [line width=1pt ] (5,0.5)--(4,0)  ;
\draw [line width=1pt ] (5,-0.5)--(4,0)  ;
\draw [line width=1pt ] (5,-0.5)--(5,0.5)  ;

\draw  [shape=circle] (7,0) circle (.1); 
\draw  [shape=circle] (8,0.5) circle (.1) ;
\draw  [shape=circle] (8,-0.5) circle (.1);
\draw  [shape=circle] (9,0) circle (.1);
\draw  [shape=circle] (10,0) circle (.1);
\draw  [shape=circle] (11,-.5) circle (.1);
\draw  [shape=circle] (11,0.5) circle (.1);
\draw  [shape=circle] (12,0) circle (.1);

\draw [line width=1pt ] (9.1,0)--(9.9,0)  ;
\draw [line width=1pt ] (7,0)--(8,.5)  ;
\draw [line width=1pt ] (7,0)--(8,-.5)  ;
\draw [line width=1pt ] (11,-.5)--(12,0)  ;
\draw [line width=1pt ] (11,.5)--(12,0)  ;

\path [pattern=north west lines, pattern  color=blue]   (8,0.5)--(8,-0.5)--(9,0)--cycle;
\path [pattern=north west lines, pattern  color=blue]   (11,0.5)--(11,-0.5)--(10,0)--cycle;

\end{tikzpicture}
\end{center}
\caption{A blow-tie and its corresponding exceptional pair of odd cycles in a labeled hypergraph.} \label{fig.bowtie}
\end{figure}
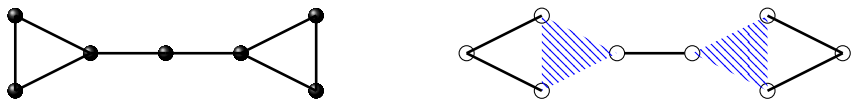
\end{example}

\begin{theorem} \label{thm.exceptional}
Let $H$ be a separated labeled hypergraph. If $H$ contains an exceptional pair of odd cycles then $\A[H] \not= k[H]$.
\end{theorem}

\begin{proof} Let $I_{H} = (x^{\a_1}, \dots, x^{\a_s})$ and assume that $(C_1,C_2)$ is an exceptional pair of odd cycles in $H$. Without loss of generality, suppose that the vertices of $C_1$ and $C_2$ are $\{1, \dots, 2p-1\}$ and $\{2p, \dots, 2q\}$ ($q > p$). Define
$$c_i = \left\{\begin{array}{ll} \dfrac{1}{2} & \text{if $i$ is a vertex in } C_1 \text{ or } C_2 \\
0 & \text{otherwise,}\end{array} \right.$$
and consider $\a = \sum_{i=1}^s c_i \a_i.$

It is easy to see that $\sum_{i=1}^s c_i = \sum_{i=1}^{2q} \dfrac{1}{2} = q \in \ZZ$. Consider an edge $E$ in $H$ and a label $w$ of $E$. If $E$ is an edge of the cycles $C_1$ and $C_2$ then the power of $w$ in $x^{\a}$ is 1. If $E$ is not an edge of $C_1$ and $C_2$ then $E$ contain an even number of vertices from $C_1$ and $C_2$, and the power of $w$ in $x^\a$ is still an integer. Thus, $\a \in \NN^n$.

Suppose now that $\a$ can be written as $\a = \sum_{i=1}^s d_i \a_i$ for $d_i \in \NN$. Consider an edge $\{l,m\}$ on the odd cycles and let $v$ be any label of this edge. Clearly, the power of $v$ in $x^\a$ is 1, and so $d_l$ and $d_m$ are exactly $1$ and $0$ (or $0$ and $1$). It follows that the coefficients $d_i$'s of the vertices on each odd cycles are $0$ and $1$ alternatively, except possibly at the simple higher dimensional edge.

Let $G_1$ and $G_2$ be the higher dimensional edges of $C_1$ and $C_2$. Since $G_1$ and $G_2$ are connected by at least an edge that does not contain any vertices from $C_1$ and $C_2$, the power of any label of edges connecting $G_1$ and $G_2$ in $x^\a$ must be 0. This implies that among the new coefficients $d_i$'s the coefficients of vertices in $G_1$ and $G_2$ that do not belong to $C_1$ and $C_2$ are still 0. Now, suppose that $l$ and $m$ are vertices on $C_1$ that belong to $G_1$. It then follows (by going around the cycle) that $d_l$ and $d_m$ are either both 0 or both 2. In particular, the power of any label $u$ of $F$ in $x^\a$ is either 0 or 2. This is a contradiction to the way $c_i$'s and $\a$ were chosen. Hence, $\A[H] \not= k[H]$ by Proposition \ref{thm.normalcondition}.
\end{proof}

\begin{figure}[h]
\begin{center}
\begin{tikzpicture}

\draw  [shape=circle] (1,0) circle (.1) ;
\draw  [shape=circle] (2,0.5) circle (.1) ;
\draw  [shape=circle] (2,-0.5) circle (.1);
\draw  [shape=circle] (3,0) circle (.1);

\draw  [shape=circle] (4,0) circle (.1);
\draw  [shape=circle] (5,-.5) circle (.1);
\draw  [shape=circle] (5,0.5) circle (.1);
\draw  [shape=circle] (6,1) circle (.1);
\draw  [shape=circle] (6,-1) circle (.1);
\draw  [shape=circle] (7,0.5) circle (.1);
\draw  [shape=circle] (7,-0.5) circle (.1);
\draw  [shape=circle] (8,0) circle (.1);

\draw [line width=1pt ] (3.1,0)--(3.9,0)  ;
\draw [line width=1pt ] (1,0)--(2,.5)  ;
\draw [line width=1pt ] (1,0)--(2,-.5)  ;
\draw [line width=1pt ] (6,-1)--(5,-.5)  ;
\draw [line width=1pt ] (6,1)--(5,.5)  ;
\draw [line width=1pt ] (6,1)--(7,.5)  ;
\draw [line width=1pt ] (6,-1)--(7,-.5)  ;
\draw [line width=1pt ] (8,0)--(7,.5)  ;
\draw [line width=1pt ] (8,0)--(7,-.5)  ;

\draw (4,-1.5) node [left] {$H_1$};

\path [pattern=north west lines, pattern  color=blue]   (2,0.5)--(2,-0.5)--(3,0)--cycle;
\path [pattern=north west lines, pattern  color=red]   (5,0.5)--(7,0.5)--(7,-0.5)--(6,-1)--cycle;

\path [pattern=north west lines, pattern  color=blue]   (5,0.5)--(5,-0.5)--(4,0)--cycle;


\draw  [shape=circle] (9,0) circle (.1) ;
\draw  [shape=circle] (10,0.5) circle (.1) ;
\draw  [shape=circle] (10,-0.5) circle (.1);
\draw  [shape=circle] (11,0) circle (.1);

\draw  [shape=circle] (12,0) circle (.1);
\draw  [shape=circle] (13,-.5) circle (.1);
\draw  [shape=circle] (13,0.5) circle (.1);
\draw  [shape=circle] (14,1) circle (.1);
\draw  [shape=circle] (14,-1) circle (.1);
\draw  [shape=circle] (15,0.5) circle (.1);
\draw  [shape=circle] (15,-0.5) circle (.1);
\draw  [shape=circle] (16,0) circle (.1);

\draw [line width=1pt ] (11.1,0)--(11.9,0)  ;
\draw [line width=1pt ] (9,0)--(10,.5)  ;
\draw [line width=1pt ] (9,0)--(10,-.5)  ;
\draw [line width=1pt ] (14,-1)--(13,-.5)  ;
\draw [line width=1pt ] (14,1)--(13,.5)  ;
\draw [line width=1pt ] (14,1)--(15,.5)  ;
\draw [line width=1pt ] (14,-1)--(15,-.5)  ;
\draw [line width=1pt ] (16,0)--(15,.5)  ;
\draw [line width=1pt ] (16,0)--(15,-.5)  ;

\draw (12,-1.5) node [left] {$H_2$};

\path [pattern=north west lines, pattern  color=blue]   (10,0.5)--(10,-0.5)--(11,0)--cycle;
\path [pattern=north west lines, pattern  color=red]   (14,1)--(15,0.5)--(15,-0.5)--(14,-1)--cycle;

\path [pattern=north west lines, pattern  color=blue]   (13,0.5)--(13,-0.5)--(12,0)--cycle;

\end{tikzpicture}
\end{center}
\caption{Hypergraphs that contain exceptional pairs of odd cycles.} \label{fig.EPairs}
\end{figure}
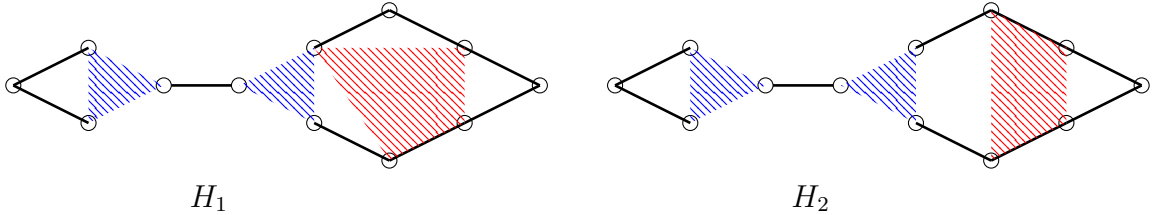

\begin{remark} Unlike the discussion for graphs in \cite{HO}, with hypergraphs, even when $\A[H] \not= k[H]$, the normality of the toric ring $k[H]$ is still a subtle question. For example, in Figure \ref{fig.EPairs} we give two labeled hypergraphs that look very similar and both contain exceptional pairs of odd cycles. The squarefree monomial ideals corresponding to these hypergraphs are:
\begin{align*}
I_{H_1} = (x_1x_2, & x_1x_3, x_2x_3, x_3x_6, x_6x_7, x_7x_{13}, x_{12}x_{13}x_{14}, x_{11}x_{12}x_{14}, \\
& x_{10}x_{11}, x_9x_{10}x_{14}, x_8x_9,x_7x_8x_{14}), \text{ and }
\end{align*}
\begin{align*}
I_{H_2} = (x_1x_2, & x_1x_3, x_2x_3, x_3x_6, x_6x_7, x_7x_{13}, x_{12}x_{13}x_{14}, x_{11}x_{12}x_{14}, \\
& x_{10}x_{11}, x_9x_{10}x_{14}, x_8x_9x_{14}, x_7x_8).
\end{align*}
Computation shows that $k[H_1]$ is normal while $k[H_2]$ is not normal.
\end{remark}


\end{document}